\documentclass[preprint,12pt]{elsarticle}

\usepackage{graphicx}
\usepackage[cp1251]{inputenc}
\usepackage{amssymb}
\usepackage{amsthm}
\usepackage{cmap}

\usepackage{amsmath}

\biboptions{sort&compress}

\newcommand{\sysn}{\left\{\begin{array}{rcl}}
\newcommand{\sysk}{\end{array}\right.}

\newtheorem{theorem}{Theorem}[section]
\newtheorem{lemma}[theorem]{Lemma}

\theoremstyle{example}

\newtheorem{proposition}[theorem]{Proposition}

\newtheorem{corollary}[theorem]{Corollary}
\theoremstyle{definition}
\newtheorem{definition}[theorem]{Definition}
\newtheorem{question}[theorem]{Question}

\journal{...}

\begin{document}

\title{On resolvability, connectedness and pseudocompactness}

\author{Anton E. Lipin}

\address{Krasovskii Institute of Mathematics and Mechanics, \\ Ural Federal
 University, Yekaterinburg, Russia}

\ead{tony.lipin@yandex.ru}

\begin{abstract} We prove that:

I. If $L$ is a $T_1$ space, $|L|>1$ and $d(L) \leq \kappa \geq \omega$, then
there is a submaximal dense subspace $X$ of $L^{2^\kappa}$ such that $|X|=\Delta(X)=\kappa$;

II. If $\frak{c}\leq\kappa=\kappa^\omega<\lambda$ and $2^\kappa=2^\lambda$, then there is a Tychonoff pseudocompact globally and locally connected space $X$ such that $|X|=\Delta(X)=\lambda$ and $X$ is not $\kappa^+$-resolvable;

III. If $\omega_1\leq\kappa<\lambda$ and $2^\kappa=2^\lambda$, then there is a regular space $X$ such that $|X|=\Delta(X)=\lambda$, all continuous real-valued functions on $X$ are constant (so $X$ is pseudocompact and connected) and $X$ is not $\kappa^+$-resolvable.
\end{abstract}


\maketitle 


\section{Introduction}

We start by recalling a central notion (E.~Hewitt, \cite{Hewitt}, 1943; J.G.~Ceder, \cite{Ceder}, 1964).
Let $X$ be a topological space and let $\kappa$ be a cardinal.
The space $X$ is called $\kappa$-{\it resolvable} if $X$ contains $\kappa$ many pairwise disjoint dense subsets.

Obviously, the necessary condition of $\kappa$-resolvability of a space is inequality $\kappa\leq|U|$ for all nonempty open subsets $U$.
The smallest cardinality of a nonempty open set in $X$ is called a {\it dispersion character} of the space $X$ and it is denoted by $\Delta(X)$.
The space $X$ is called $\kappa$-{\it irresolvable} if $\kappa\leq\Delta(X)$ and $X$ is not $\kappa$-resolvable.
The space $X$ is called {\it maximally resolvable} if $X$ is $\Delta(X)$-resolvable.
If $\kappa=2$, then instead of $2$-resolvable and $2$-irresolvable we say {\it resolvable} and {\it irresolvable}, respectively.

We refer the reader to a selective survey \cite{survey} made by O.~Pavlov in 2007 for more detailed information on resolvability.

In this paper, we focus on two properties whose relation to resolvability is mostly unknown.
We show that even together these two properties do not entail maximal resolvability~--- at least, in ZFC and in the class of Tychonoff spaces.

The first property is connectedness. Hausdorff examples of irresolvable connected spaces were constructed by many authors.
To the author's knowledge, the earliest one, published in 1953, belongs to K.~Padmavally \cite{Padmavally}.
The question whether every infinite regular connected space is resolvable remains open, although something is known for local connectedness
(recall that a space is said to be {\it locally connected} if it has a base, consisting of connected sets).

C.~Costantini proved in 2005 that every crowded (i.e. without isolated points) regular locally connected space is $\omega$-resolvable \cite{Costantini}.
In \cite{Pavlov} Pavlov writes that I.V.~Yaschenko proved that every crowded locally connected Tychonoff space is $\frak{c}$-resolvable.
A.~Dehghani and M.~Karavan proved in 2015 that every crowded locally connected functionally Hausdorff space is $\frak{c}$-resolvable \cite{DK}.

The second property we focus on in this paper is pseudocompactness.
Recall that a space $X$ is called pseudocompact if every continuous real-valued function on $X$ is bounded.
Two different examples of a Hausdorff pseudocompact (actually, even countably compact) irresolvable space were published by V.I.~Malykhin in 1998 \cite{Malykhin} and by Pavlov in 2001 \cite{Pavlov}.

For Tychonoff pseudocompact spaces some positive results were obtained in 2016-2018.
J.~van Mill proved that a crowded Tychonoff pseudocompact space $X$ is $\frak{c}$-resolvable if $c(X)=\omega$ (where $c(X)$ is the cellularity of $X$) \cite{Mill}.
Then Y.F. Ortiz-Castillo and A.H. Tomita proved that a crowded Tychonoff pseudocompact space $X$ is resolvable if $c(X)\leq\frak{c}$ \cite{OT}.
Finally, I. Juhasz, L. Soukup and Z. Szentmiklossy published a theorem, from which it follows that a crowded Tychonoff pseudocompact space $X$ is $\frak{c}$-resolvable if $c(X)<\frak{c}^{+\omega}$.
Moreover, it is consistent that every crowded Tychonoff pseudocompact space $X$ is $\frak{c}$-resolvable and the consistency of the negation of this requires large cardinals \cite{JSSz}.
The question whether resolvability of a crowded Tychonoff pseudocompact space can be proved in ZFC remains open.

We also refer the reader to \cite{Pytkeev, Lipin} for some results on resolvability of countably compact spaces.

Our choice of these two properties may seem strange, but the point is that both properties $\mathcal{P}\in\{\text{connectedness}, \text{pseudocompactness}\}$
satisfy the following condition: if $X$ is a space and some dense subset of $X$ has $\mathcal{P}$, then $X$ has $\mathcal{P}$.

This observation leads us to the following idea of an example.
Let us take some space $X$ and find two dense subspaces $C,E \subseteq X$ such that $C$ is connected pseudocompact and $|C|=\kappa$,
whereas $E$ is irresolvable and $\Delta(E)>\kappa$.
Clearly, the subspace $C \cup E$ is connected pseudocompact and $\kappa^+$-irresolvable.

We suppose that the most natural choice of $X$ here is a Tychonoff cube.
We use it in section 4, constructing a consistent example of a Tychonoff pseudocompact globally and locally connected $\frak{c}^+$-irresolvable space.
Our method cannot give less than $\frak{c}$-resolvable Tychonoff space (connected or pseudocompact), but in section 5 we construct a consistent with $\neg \mathrm{CH}$ example of a regular pseudocompact connected $\omega_2$-irresolvable space. This time we choose as $X$ some power of the space constructed in 2016 by K.C.~Ciesielski and J.~Wojciechowski \cite[Theorem 7(i,iii)]{CW}.

The construction of $C$ is relatively easy modulo Hewitt-Marczewski-Pondiczery theorem \cite[Theorem 2.3.15]{Engelking}.

The construction of $E$ requires a bit more work. We do it in section 3. Actually, $E$ will be submaximal (which means that all dense subsets of $E$ are open). Our technique can be briefly described as a recursive determination of coordinates in $X$ of all points of $E$. We refer the reader to \cite{ATTW, JSSz-1, CH-1, CH} for other examples of this and similar methods in resolvability. In particular, in \cite{JSSz-1} a dense submaximal subspace is constructed in every Cantor cube of the form $\{0,1\}^{2^\kappa}$. As a corollary, the authors of \cite{JSSz-1} obtain a countable dense submaximal subspace of $[0,1]^\frak{c}$.
It is not clear whether their technique can be used to construct an uncountable dense submaximal subspace in any Tychonoff cube.

\section{Preliminaries}

We assume the following notation and conventions.

\begin{itemize}

\item Symbol $\bigsqcup$ denotes {\it disjoint union} in the following sense: its equal to the usual union, but using it we assume that the united sets are pairwise disjoint.

\item We use letters $\kappa,\lambda,\mu$ for cardinals and $\alpha,\beta,\gamma$ for ordinals.

\item If $f$ is a function, then we write $f[A]$ for $\{f(x) : x\in A\}$. It is essential if $A$ is both a subset and an element of the domain of $f$.

\item If $f,g$ are functions, then $f \circ g$ is a composition $(f \circ g)(x) = g(f(x))$.

\item If $f : X \to Y$ and $A \subseteq X$, then $f|_A$ is a restriction of $f$ to $A$.

\item {\it Space} means topological space.

\item $\Delta(X)$ is the dispersion character of the space $X$, i.e. the smallest cardinality of a nonempty open subset of $X$.

\item If $A$ is a subset of some space, then $\overline{A}$ is the closure of $A$.

\item A set $A$ is called $\kappa$-{\it dense} in a space $X$ if $|A \cap U|\geq\kappa$ holds for all nonempty open $U \subseteq X$.

\item $d(X)$ is the {\it density} of the space $X$, i.e. the smallest cardinality of a dense subset of $X$.

\end{itemize}

We also need the following formal ``improvement'' of the Hewitt-Marczewski-Pondiczery theorem.
In fact, this ``improvement'' is obtained in the classic proof (as in \cite[Theorem 2.3.15]{Engelking}).

\begin{theorem}\label{T_HMP}
If $X$ is a space, $|X|>1$ and $d(X)\leq\kappa\geq\omega$, then there is a set $A \subseteq X^{2^\kappa}$ of cardinality $\kappa$ such that:
\begin{itemize}
\item[(1)] $A$ is $\kappa$-dense in $ X^{2^\kappa}$;

\item[(2)] for every $x \in A$ the set $\{x(\alpha) : \alpha < 2^\kappa\}$ is finite.
\end{itemize}
\end{theorem}

\section{Submaximal subspaces in topological cubes}

Recall that a crowded space $X$ is called {\it submaximal} if every dense subset is open in $X$. Obviously, all submaximal spaces are irresolvable. Our goal in this section is to prove Theorem \ref{T_submax} to obtain a class of submaximal spaces which will be convenient for our further purposes.
Let us fix some additional notation.

\begin{itemize}
\item If $X$ is a space and $\kappa,\lambda$ are cardinals, then we denote $X^{\kappa\oplus\lambda} := X^\kappa\times X^\lambda$.

\item Suppose we are given a space $X$ and functions $f_\alpha : X \to Y_\alpha$ for all $\alpha<\gamma$. We denote by $\nabla_{\alpha<\gamma}f_\alpha$
the function $F : X \to X \times \prod_{\alpha<\gamma}Y_\alpha$ such that $F(x) = (x, \{f_\alpha(x)\}_{\alpha<\gamma})$ for all $x \in X$ (so, $F$ is a map from $X$ to the graph of the diagonal product of the functions $f_\alpha$).
\end{itemize}

The following proposition is just a reformulation of Theorem \ref{T_HMP}(1), obtained by a permutation of coordinates.

\begin{proposition}\label{P_HMP}
If $Q$ is a space, $|Q|>1$ and $d(Q)\leq\kappa\geq\omega$, then there is a set $A \subseteq Q^\kappa$ and functions $f_\alpha:A \to Q$ for $\alpha<2^\kappa$ such that $|A|=\kappa$ and the set $(\nabla_{\alpha<2^\kappa}f_\alpha)[A]$ is $\kappa$-dense in $Q^{\kappa\oplus{2^\kappa}}$.
\end{proposition}

\begin{proof}
Take any $\kappa$-dense set $B$ of cardinality $\kappa$ in $Q^{2^\kappa}$.
Choose any set $M \subseteq 2^\kappa$ of cardinality $\kappa$ with the property that for any different $x,y \in B$ there is an ordinal $\alpha \in M$ such that $x(\alpha) \ne y(\alpha)$.
Denote by $N$ the set $2^\kappa \setminus M$.
Choose any bijections $\varphi : \kappa \leftrightarrow M$ and $\psi : 2^\kappa \leftrightarrow N$.
Define $A := \{\varphi \circ x : x \in B\}$.
It is obvious that $\varphi \circ x \ne \varphi \circ y$ for all different $x,y \in B$.
For every $\alpha < 2^\kappa$ define the function $f_\alpha : A \to Q$ is such a way that $f_\alpha(\varphi \circ x) = x(\psi(\alpha))$ for all $x \in B$.
It is clear that the set $A$ and the functions $f_\alpha$ are as required.
\end{proof}

We also recall the following fact.

\begin{proposition}\label{P_submax}
If $X$ is a $T_1$ space and all $\Delta(X)$-dense subsets of $X$ are open, then $X$ is submaximal.
\end{proposition}
\begin{proof}
Let $D$ be any subset of $X$ with cardinality less than $\Delta(X)$. Clearly, the set $E := X \setminus D$ is $\Delta(X)$-dense, hence open, so $D$ is nowhere dense.
So, all dense subsets of $X$ are $\Delta(X)$-dense, consequently open.
\end{proof}

\begin{theorem}\label{T_submax}
If $L$ is a $T_1$ space, $|L|>1$ and $d(L) \leq \kappa \geq \omega$, then
there is a submaximal $\kappa$-dense subspace with cardinality $\kappa$ in $L^{2^\kappa}$.
\end{theorem}
\begin{proof}
First of all, we can suppose that $L$ is crowded (hence infinite), because otherwise we just redesignate $L$ to be $L^\omega$.
Choose any $r \in L$ and denote $Q := L \setminus \{r\}$.
Let us take a set $A \subseteq Q^\kappa$ and functions $f_\alpha:A\to Q$ for all $\alpha<2^\kappa$ with the properties from Proposition \ref{P_HMP}.
Also, let $\{A_\alpha : \alpha<2^\kappa\}$ be a family of all subsets of $A$.

Now let us define functions $g_\alpha : A \to L$ for all $\alpha<2^\kappa$ using recursion by $\alpha$.
Suppose $g_\beta$ are defined for all $\beta<\alpha$.
Denote by $F_\alpha$ the function $\nabla_{\beta<2^\kappa} h_\beta$, where $h_\beta=g_\beta$ if $\beta<\alpha$ and $h_\beta=f_\beta$ if $\alpha\leq\beta<2^\kappa$.

If the set $F_\alpha[A_\alpha]$ is $\kappa$-dense in $L^{\kappa\oplus{2^\kappa}}$, we define $g_\alpha(x) := f_\alpha(x)$ for $x \in A_\alpha$ and $g_\alpha(x) := r$ for $x \in A \setminus A_\alpha$. Otherwise, if the set $F_\alpha[A_\alpha]$ is not $\kappa$-dense in $Q^{\kappa\oplus{2^\kappa}}$, we define $g_\alpha := f_\alpha$.

The functions $g_\alpha$ are constructed. Define $F = F_{2^\kappa} := \nabla_{\alpha<2^\kappa}g_\alpha$. We shall prove that the subspace $X := F[A]$ is the required one (up to a homeomorphism between $L^{2^\kappa}$ and $L^{\kappa\oplus 2^\kappa}$).
We need the following notation:

\begin{itemize}

\item $\mathcal{W}$ is the standard base of $L^{\kappa\oplus{2^\kappa}}$, i.e. a family of all nonempty sets $W=V\times\prod_{\gamma<2^\kappa}U_\gamma$, where $V$ is open in $L^\kappa$, all $U_\gamma$ are open in $L$ and the set $\Gamma_W := \{\gamma<2^\kappa : U_\gamma \ne L\}$ is finite;

\item $\mathcal{W^-}$ is the family of all $W=V\times\prod_{\gamma<2^\kappa}U_\gamma\in\mathcal{W}$ such that $r \notin U_\gamma$ for all $\gamma\in\Gamma_W$.
Clearly, this is a $\pi$-base of $L^{\kappa\oplus{2^\kappa}}$;

\item for every $\alpha \leq 2^\kappa$ we denote by $\mathcal{D}_\alpha$ the family of all $S \subseteq A$ such that the set $F_\alpha[S]$ is $\kappa$-dense in $L^{\kappa\oplus{2^\kappa}}$.

\end{itemize}

We finish the proof in the following sequence of claims.

\begin{itemize}

\item[I.] For every $x \in A$, $\alpha \leq 2^\kappa$ and $W = V\times\prod_{\gamma<2^\kappa}U_\gamma \in \mathcal{W}$ the following conditions are equivalent:

\begin{itemize}

\item[(1)] $F_\alpha(x) \in W$;

\item[(2)] $x \in V$, $g_\gamma(x) \in U_\gamma$ if $\gamma < \alpha$ and $f_\gamma(x) \in U_\gamma$ if $\alpha \leq \gamma < 2^\kappa$.

\end{itemize}

\item[II.] $X$ is $\kappa$-dense in $L^{\kappa\oplus{2^\kappa}}$ (or, in other words, $A \in \mathcal{D}_{2^\kappa}$).

We prove that $A \in \mathcal{D}_\alpha$ for all $\alpha \leq 2^\kappa$ by induction on $\alpha$.
The base $\alpha=0$ holds trivially by the choice of $A$ and $f_\alpha$.

Now, if $A \in \mathcal{D}_\alpha$, then there is a set $S \subseteq A$ (which can be $A_\alpha$ or $A$ depending on the case in the definition of $g_\alpha$) such that $F_\alpha[S]$ is $\kappa$-dense in $L^{\kappa\oplus{2^\kappa}}$ and $g_\alpha|_S = f_\alpha|_S$ (hence $F_{\alpha+1}|_S = F_\alpha|_S$). Consequently, $A \in \mathcal{D}_{\alpha+1}$.

Finally, let $\alpha$ be limited and let $A \in \mathcal{D}_\beta$ for all $\beta<\alpha$.
Take any $W=V\times\prod_{\gamma<2^\kappa}U_\gamma \in \mathcal{W}$.
Denote $\beta := \max(\Gamma_W \cap \alpha)+1$ (we define $\max\emptyset := 0$).
Clearly, $\beta<\alpha$, so $|F_\beta[A]\cap W|=\kappa$ by the inductive assumption.
Let us show that $F_\alpha[A]\cap W = F_\beta[A]\cap W$,
i.e. that the conditions $F_\alpha(x) \in W$ and $F_\beta(x) \in W$ are equivalent for every $x \in A$.
By claim I, we must only prove that $g_\gamma(x) \in U_\gamma$ if and only if $f_\gamma(x) \in U_\gamma$ for $\beta \leq \gamma < \alpha$.
But for $\beta \leq \gamma < \alpha$ we have $U_\gamma = L$ by the choice of $\beta$,
so both conditions $g_\gamma(x) \in U_\gamma$ and $f_\gamma(x) \in U_\gamma$ hold.
So, $|F_\alpha[A]\cap W| = |F_\beta[A]\cap W| = \kappa$, hence $A \in \mathcal{D}_\alpha$.

\item[III.] If $\alpha\leq\beta\leq2^\kappa$, then $\mathcal{D}_\alpha \supseteq \mathcal{D}_\beta$.

Let us prove that for any $S \subseteq A$ and $W=V\times\prod_{\gamma<2^\kappa}U_\gamma \in \mathcal{W^-}$ we have $F_\beta[S] \cap W \subseteq F_\alpha[S] \cap W$.
We take any point $F_\beta(x) \in F_\beta[S] \cap W$.
Clearly, $x\in S\cap V$, $g_\gamma(x) \in U_\gamma$ for $\gamma < \beta$ and $f_\gamma(x) \in U_\gamma$ for $\beta \leq \gamma < 2^\kappa$.
Since $W \in \mathcal{W}^-$, we have $r \notin U_\gamma$ for all $\gamma \in \Gamma_W$.
This means that for all $\gamma \in \Gamma_W\cap\beta$ we have $g_\gamma(x) \ne r$, so $g_\gamma(x) = f_\gamma(x)$.
So, $F_\alpha(x) \in W$ by the condition (2) in claim I.

Considering that $\mathcal{W}^-$ is a $\pi$-base of $L^{\kappa\oplus{2^\kappa}}$, it follows that if $F_\beta[S]$ is $\kappa$-dense in $L^{\kappa\oplus{2^\kappa}}$, then $F_\alpha[S]$ is $\kappa$-dense in $L^{\kappa\oplus{2^\kappa}}$,
i.e. $\mathcal{D}_\alpha \supseteq \mathcal{D}_\beta$.

\item[IV.] Every $\kappa$-dense subset of $X$ is open in $X$.

Let $D$ be a $\kappa$-dense set in $X$ (hence in $L^{\kappa\oplus{2^\kappa}}$).
Take $\alpha<2^\kappa$ such that $D=F[A_\alpha]$.
So, $A_\alpha \in \mathcal{D}_{2^\kappa}$.
Hence, by claim III we have $A_\alpha \in \mathcal{D}_\alpha$.
By the definition of $g_\alpha$, it means that $A_\alpha = g_\alpha^{-1}[Q]$.
Denote by $\pi_\alpha : X \to L$ the function defined by the rule $\pi_\alpha(F(x)) := g_\alpha(x)$ for all $x \in A$
(so it is a projection of $X$ onto the coordinate $\alpha$).
Clearly, the function $\pi_\alpha$ is continuous and $D = \pi_\alpha^{-1}[Q]$, so $D$ is open in $X$.

\end{itemize}

So, $X$ is submaximal by Proposition \ref{P_submax}.
\end{proof}

\section{A Tychonoff pseudocompact connected $\frak{c}^+$-irresolvable space}

\begin{definition}
Let $A \subseteq [0,1]^\lambda$. Denote by $I{^1}(A)$ the union of all line segments with ends in $A$. For every natural $n$ we denote $I^{n+1}(A)=I^1(I^n(A))$ and $I^\omega(A)=\bigcup_{n\in\omega}I^n(A)$.
\end{definition}

\begin{proposition}
If a set $A$ is dense in $[0,1]^\lambda$ and $I^\omega(A) \subseteq B \subseteq [0,1]^\lambda$, then the subspace $B$ is globally and locally connected.
\end{proposition}

\begin{proposition}
For every $A \subseteq [0,1]^\lambda$ there is a countably compact subspace $B \subseteq [0,1]^\lambda$ such that $B \supseteq A$ and $|B| \leq |A|^\omega$.
\end{proposition}
\begin{proof}
For every infinite countable set $S \subseteq [0,1]^\lambda$ choose any limited point $x(S)$.
For every set $H \subseteq [0,1]^\lambda$ denote $L_1(H) := H \cup \{x(S) : S\subseteq H, |S|=\omega\}$, $L_{\alpha+1}(H) := L_1(L_\alpha(H))$ for all ordinals $\alpha$ and $L_\gamma(H) := \bigcup_{\alpha<\gamma} L_\alpha(H)$ for all limited ordinals $\gamma$.
It is clear that $B := L_{\omega_1}(A)$ is the required subspace.
\end{proof}

\begin{theorem}\label{T_Tych}
If $\frak{c}\leq\kappa=\kappa^\omega<\lambda$ and $2^\kappa=2^\lambda=\mu$, then there is a $\kappa^+$-irresolvable globally and locally connected pseudocompact space $X$ such that $|X|=\Delta(X)=\lambda$ and $X$ is homeomorphic to a dense subset of $[0,1]^\mu$.
\end{theorem}
\begin{proof}
By Theorem \ref{T_HMP}, there is a dense subset $A$ of $[0,1]^\mu$ such that $|A|=\kappa$.
Choose any countably compact subspace $C \supseteq I^\omega(A)$ in $[0,1]^\mu$ such that $|C|=\kappa$.
Denote $E$ any $\lambda$-dense irresolvable subspace of cardinality $\lambda$ in $[0,1]^\mu$.
The subspace $X := C \cup E$ is the required one.
\end{proof}

\begin{corollary}
ZFC does not prove the hypothesis ``If a Tychonoff pseudocompact space $X$ is globally and locally connected and $\Delta(X)>\frak{c}$, then $X$ is $\frak{c}^+$-resolvable".
\end{corollary}
\begin{proof}
We take any model with $2^{(\frak{c}^+)}=2^\frak{c}$ and apply Theorem \ref{T_Tych} to $\kappa= \frak{c}$ and $\lambda=\frak{c}^+$.
\end{proof}

\section{A regular pseudocompact connected $\omega_2$-irresolvable space}

The following statement can be found in \cite[Theorem 7(i,iii)]{CW}.

\begin{theorem}\label{T_CW}
For every uncountable cardinal $\kappa$ there is a regular space $L$ such that $|L|=\kappa$ and all continuous real-valued functions on $L$ are constant.
\end{theorem}

It is obvious that any such $L$ is connected and pseudocompact. We construct our example in some power of $L$. But first we need to do some preperations. The complexity comparing to the previous section is that now we cannot obtain a connected subspace by connecting all pairs of points with line segments.

\begin{definition}
Let $X$ be a space and let $\lambda$ be a cardinal. We say that:

\begin{itemize}

\item a point $x \in X^\lambda$ is a {\it stair-point} (of {\it range} $n\in\omega$),
if the set $\{x(\alpha) : \alpha<\lambda\}=x[\lambda]$ is finite (is of cardinality $n$);

\item a family $\mathcal{E}$ of subsets of $X^\lambda$ is a {\it web} in $X^\lambda$ if all elements of $\mathcal{E}$ are homeomorphic to $X$ and for all $A,B \in \mathcal{E}$ there are sets $C_1, \ldots, C_n \in \mathcal{E}$ such that $C_1 = A$, $C_n = B$ and $C_k \cap C_{k+1} \ne \emptyset$ for all $k < n$.

\end{itemize}
\end{definition}

\begin{lemma}
Let $X$ be a space and let $\lambda$ be a cardinal. For any two stair-points $a,b \in X^\lambda$ there is a finite web $\mathcal{E}$ in $X^\lambda$ such that $a,b \in \bigcup\mathcal{E}$.
\end{lemma}
\begin{proof}
For any $x \in X^\lambda$, $\Gamma\subseteq\lambda$ and $p\in X$ denote by $f(x,\Gamma,p)$ the point $y\in X^\lambda$ such that
$y(\alpha)=x(\alpha)$ for $\alpha\notin\Gamma$ and $y(\alpha)=p$ for $\alpha\in\Gamma$.
Construct $N(x,\Gamma) := \{f(x,\Gamma,p) : p \in X\}$. Clearly, the subspace $N(x,\Gamma)$ is homeomorphic to $X$ if $\Gamma\ne\emptyset$.

Denote by $n$ the range of the stair-point $a$.
Now we define points $a_n, \ldots, a_1$ and sets $A_{n-1}, \ldots, A_1$ in the following manner.
Firstly, $a_n := a$.
Now, let $a_k$ be a stair-point of range $k>1$.
Choose any different $p_k, q_k \in a_k[\lambda]$ and denote $\Gamma_k := a_k^{-1}(p_k)$.
Define $A_{k-1} := N(a_k, \Gamma_k)$ and $a_{k-1} := f(a_k, \Gamma_k, q_k)$.
It is clear that $a_{k-1}$ is a stair-point of range $k-1$ and $A_{k-1} \supseteq \{a_k, a_{k-1}\}$.

Now we denote $m$ the range of the stair-point $b$ and choose points $b_m, \ldots, b_1$ and sets $B_{m-1}, \ldots, B_1$ in the same manner, so every $b_k$ is a stair-point of range $k$, the set $B_k$ is homeomorphic to $X$ and $B_{k-1} \supseteq \{b_k, b_{k-1}\}$.

Finally, denote by $D$ the diagonal of $X^\lambda$, which is also the set of all stair-points of range $1$.

It is clear that the family $\mathcal{E} := \{A_{n-1}, \ldots, A_1, D, B_1, \ldots, B_{m-1}\}$ is the required web.
\end{proof}

\begin{lemma}
For every space $X$ and any cardinal $\kappa\geq\omega$ there is a web $\mathcal{E}$ in $X^{2^\kappa}$ such that the set $\bigcup \mathcal{\mathcal{E}}$ is dense in $X^{2^\kappa}$ and $|\mathcal{E}| \leq \kappa$.
\end{lemma}
\begin{proof}
By Theorem \ref{T_HMP}(2), there is a dense in $X^{2^\kappa}$ set $A$ of cardinality $\kappa$ such that every point in $A$ is a stair-point.
By the previous lemma, for all $x,y \in A$ we can choose a finite web $\mathcal{E}_{x,y}$ in $X^{2^\kappa}$ such that $x,y\in \bigcup\mathcal{E}_{x,y}$. Clearly, the family $\mathcal{E} := \bigcup_{x,y \in A} \mathcal{E}_{x,y}$ is the required web.
\end{proof}

\begin{theorem}\label{T_Reg}
If $\omega_1\leq\kappa<\lambda$ and $2^\kappa=2^\lambda$, then there is a $\kappa^+$-irresolvable regular space $X$ such that $|X|=\Delta(X)=\lambda$ and all continuous real-valued functions on $X$ are constant.
\end{theorem}
\begin{proof}
Denote $\mu := 2^\kappa$.
Take any regular space $L$ of cardinality $\kappa$ with the property that all continuous real-valued functions on $L$ are constant.
Take any web $\mathcal{E}$ in $L^\mu$ such that the set $C := \bigcup \mathcal{\mathcal{E}}$ is dense in $L^\mu$ and $|\mathcal{E}| \leq \kappa$.
Clearly, $|C|=\kappa$ and all continuous real-valued functions on $C$ are constant.
Denote by $E$ any $\lambda$-dense irresolvable subspace of cardinality $\lambda$ in $L^\mu$.
The subspace $X := C \cup E$ is the required one.
\end{proof}

\begin{corollary}\label{C_Reg}
ZFC does not prove the hypothesis ``If a regular pseudocompact space $X$ is connected and $\Delta(X)>\frak{\omega_1}$, then $X$ is $\omega_2$-resolvable".
\end{corollary}
\begin{proof}
We take any model with $2^{\omega_2}=2^{\omega_1}$ and apply Theorem \ref{T_Reg} to $\kappa= \omega_1$ and $\lambda=\omega_2$.
\end{proof}

\section{Questions}

\begin{question}
Is it consistent that every Tychonoff (or even regular) space with a property $\mathcal{P}$ is maximally resolvable, if:

\begin{itemize}

\item[(1)] $\mathcal{P}$ is pseudocompactness?

\item[(2)] $\mathcal{P}$ is connectedness?

\item[(3)] $\mathcal{P}$ is local connectedness?

\item[(4)] $\mathcal{P}$ is any combination of (1-3)?

\item[(5)] $\mathcal{P}$ is any combination of (1-3) and GCH holds?

\end{itemize}

\end{question}

\begin{question}
Suppose $\neg\mathrm{CH}$. Is there a crowded regular locally connected space of cardinality less than $\frak{c}$?
\end{question}

And, of course, two of the most intriguing questions (posed in \cite{survey, CG}) remain open:

\begin{itemize}

\item Is every infinite regular (or Tychonoff) connected space resolvable?

\item Is every crowded regular (or Tychonoff) pseudocompact space resolvable in ZFC?

\end{itemize}


\medskip


\bibliographystyle{model1a-num-names}
\bibliography{<your-bib-database>}

\end{document}